\documentclass[12pt]{amsart}
\usepackage{amsfonts,amsmath,amsthm,amssymb}
\usepackage{latexsym}
\usepackage{graphics}
\newtheorem{theorem}{Theorem}[section]
\newtheorem{corollary}[theorem]{Corollary}

\newtheorem{lemma}[theorem]{Lemma}

\newtheorem{proposition}[theorem]{Proposition}

\theoremstyle{definition}

\theoremstyle{remark}

\theoremstyle{remark}

\newcommand{\beql}[1]{\begin{equation}\label{#1}}
\newcommand{\eeq}{\end{equation}}

\begin{document}

\title[Complex Hadamard matrices and equiangular tight frames]{Complex Hadamard matrices and equiangular tight frames}

\author{Ferenc Sz\"oll\H{o}si}

\date{April, 2011.}

\address{Ferenc Sz\"oll\H{o}si: Department of Mathematics and its Applications, Central European University, H-1051, N\'ador u. 9, Budapest, Hungary}\email{szoferi@gmail.com}

\thanks{This work was supported by the Hungarian National Research fund OTKA K-77748}

\begin{abstract}
In this paper we give a new construction of parametric families of complex Hadamard matrices of square orders, and connect them to equiangular tight frames. The results presented here generalize some of the recent ideas of Bodmann \emph{et al}.\ \cite{Bodmann} and extend the list of known equiangular tight frames. In particular, a $(36,21)$-frame coming from a nontrivial cube root signature matrix is obtained for the first time.
\end{abstract}

\maketitle

{\bf 2000 Mathematics Subject Classification.} Primary 05B20, secondary 46L10.
	
{\bf Keywords and phrases.} {\it Complex Hadamard matrices, Equiangular tight frames}

\section{Introduction}\label{sec:intro}
Let $\mathcal{H}$ be a complex Hilbert space, and let $F=\{f_i\}_{i\in I}\subset \mathcal{H}$ be a subset. We call $F$ a frame for $\mathcal{H}$ provided that there are two constants $C,D>0$ such that the inequality
\[C\left\|x\right\|^2\leq \sum_{i\in I}\left|\left\langle x,f_i\right\rangle\right|^2\leq D\left\|x\right\|^2\]
holds for every $x\in\mathcal{H}$. When $C=D=1$ then the frame is called normalized, tight (also called a Parseval frame). Real and complex equiangular tight frames have been heavily investigated recently \cite{BodmannPAMS,Bodmann,four,F,God2,R} as they have applications in coding \cite{HP} and quantum information theory \cite{grassl}. Throughout this paper we shall be concerned only with Parseval frames for the $k$ dimensional complex Hilbert space, equipped with the usual inner product. We use the term $(n,k)$ frame to refer to a Parseval frame consisting of $n$ vectors in $\mathbb{C}^k$. Every such Parseval frame induces an isometric embedding of $\mathbb{C}^k$ into $\mathbb{C}^n$ with the map
\[ V: \mathbb{C}^k\mapsto \mathbb{C}^n, (Vx)_j=\left\langle x,f_j\right\rangle, \text{for all } 1\leq j\leq n\]
which is called the analysis operator of the frame. As $V$ is linear, we can identify it with an $n\times k$ matrix and the frame vectors $\{f_1,\hdots, f_n\}$ are the respective columns of $V^\ast$. A standard argument shows \cite{Bodmann, HP} that an $(n,k)$ frame is determined up to a natural unitary equivalence by its Gram matrix $VV^\ast$, which is a self-adjoint projection of rank $k$, moreover, if the frame is uniform and equiangular (i.\ e.\ $\|f_i\|^2$ and $\left|\left\langle f_i,f_j\right\rangle\right|$ are constants for all $1\leq i\leq n$ and for all $i\neq j$, $1\leq i,j\leq n$, respectively), it follows that
\[VV^\ast=\frac{k}{n}I_n+\sqrt{\frac{k(n-k)}{n^2(n-1)}}Q\]
where $Q$ is a self-adjoint matrix with vanishing diagonal and unimodular off-diagonal entries, and $I_n$ is the identity matrix of order $n$. We refer to such frames as ETFs. The matrix $Q$ is called the Seidel matrix \cite{S1} or signature matrix associated with the $(n,k)$ frame. Note that $n\leq k^2$ and it is an outstanding open problem to decide if the upper bound is reachable for every $k$ \cite{grassl}.

The other mathematical objects concerned in this paper are complex Hadamard matrices (CHMs). A matrix $H$ with unimodular entries satisfying $HH^\ast=nI_n$ called complex Hadamard. Complex Hadamard matrices have important applications in many areas of mathematics, including algebraic combinatorics \cite{GC}, operator theory \cite{popa} and harmonic analysis \cite{KM, tao}.

The purpose of this paper is to directly relate the two objects introduced to each other and establish a connection between complex Hadamard matrices and a class of equiangular tight frames. The correspondence we developed here turned out to be quite fruitful: on the one hand we construct new, previously unknown parametric families of self-adjoint complex Hadamard matrices of square orders, on the other hand we exhibit (amongst others) a $36\times 36$ nontrivial cube root signature matrix leading to an equiangular $(36,21)$ frame which settles a recent question of Bodmann \emph{et al.}\ \cite{Bodmann}.
\section{CHMs and ETFs}
We recall the following elegant characterization of complex equiangular frames \cite{HP, S1, S2}.
\begin{theorem}[Holmes--Paulsen, \cite{HP}]
Let $Q$ be a self-adjoint $n\times n$ matrix with $Q_{ii}=0$ and $\left|Q_{ij}\right|=1$ for all $i\neq j$. Then the following are equivalent:
\begin{enumerate}
\item[(a)] $Q$ is the signature matrix of an equiangular $(n,k)$ frame for some $k$;
\item[(b)] $Q^2=(n-1)I+\mu Q$ for some necessarily real $\mu$; and
\item[(c)] $Q$ has exactly two eigenvalues.
\end{enumerate}
\end{theorem}
Note that the value $k$ above depends on $n$ and $\mu$ only. In particular (see \cite{Bodmann}), we have
\beql{kval}
k=\frac{n}{2}-\frac{\mu n}{2\sqrt{4(n-1)+\mu^2}}.
\eeq
It follows that if $Q$ is a signature matrix of an $(n,k)$ frame, then $-Q$ is a signature matrix of an $(n,n-k)$ frame. Let us recall here a quick example from \cite{Bodmann}:
\beql{Q9}
Q_9=\left[
\begin{array}{ccccccccc}
 0 & 1 & 1 & 1 & 1 & 1 & 1 & 1 & 1 \\
 1 & 0 & 1 & \omega  & \omega  & \omega  & \omega ^2 & \omega ^2 & \omega ^2 \\
 1 & 1 & 0 & \omega ^2 & \omega ^2 & \omega ^2 & \omega  & \omega  & \omega  \\
 1 & \omega ^2 & \omega  & 0 & \omega  & \omega ^2 & 1 & \omega  & \omega ^2 \\
 1 & \omega ^2 & \omega  & \omega ^2 & 0 & \omega  & \omega  & \omega ^2 & 1 \\
 1 & \omega ^2 & \omega  & \omega  & \omega ^2 & 0 & \omega ^2 & 1 & \omega  \\
 1 & \omega  & \omega ^2 & 1 & \omega ^2 & \omega  & 0 & \omega ^2 & \omega  \\
 1 & \omega  & \omega ^2 & \omega ^2 & \omega  & 1 & \omega  & 0 & \omega ^2 \\
 1 & \omega  & \omega ^2 & \omega  & 1 & \omega ^2 & \omega ^2 & \omega  & 0
\end{array}
\right], \omega=\mathbf{e}^{2\pi\mathbf{i}/3}.
\eeq
The matrix $Q_9$ above is a $9\times 9$ nontrivial cube root signature matrix of an equiangular $(9, 6)$-frame. One can observe immediately, that $Q_9$ is ``almost'' a complex Hadamard matrix. Indeed, $H=Q_9+I$ is complex Hadamard. Now we exploit this connection in general, and give a characterization of some special signature matrices through complex Hadamard matrices. The main result of this paper is the following
\begin{theorem}\label{TH}
Let $Q$ be a self-adjoint $n\times n$ matrix with $Q_{ii}=0$ and $\left|Q_{ij}\right|=1$ for all $i\neq j$. Then the following are equivalent:
\begin{enumerate}
\item[(a)] $Q^2=(n-1)I+\mu Q$ for some necessarily real $-2\leq\mu\leq 2$; and
\item[(b)] $H:=Q+\lambda I$ is a complex Hadamard matrix for $\lambda=-\mu/2\pm\mathbf{i}\sqrt{1-\left|\mu\right|^2/4}$.
\end{enumerate}
\end{theorem}
\begin{proof}
Suppose that we have a signature matrix $Q$ satisfying (a). Then, as the number $\lambda$ defined in (b) is unimodular we find that $H=Q+\lambda I$ is a unimodular matrix, moreover we have
\[HH^\ast=(Q+\lambda I)(Q^\ast+\overline{\lambda}I)=nI+\mu Q+2\Re[\lambda]Q=nI,\]
as required. Conversely, let us suppose that we have a complex Hadamard matrix $H$ with constant diagonal $\lambda$, such that the matrix $Q:=H-\lambda I$ is self-adjoint. Then, we find that
\[Q^2=QQ^\ast=(H-\lambda I)(H^\ast-\overline{\lambda}I)=(n+1)I-\lambda(Q+\lambda I)^\ast-\overline{\lambda}(Q+\lambda I)=(n-1)I-2\Re[\lambda]Q,\]
and, as $\lambda$ is unimodular, the statement (a) follows.
\end{proof}
In \cite{BodmannPAMS, HP} various examples are given indicating that complex Hadamard matrices are leading to equiangular tight frames, but so far no general theory has been developed. It is well-known and easy, see \cite{Bodmann}, that with tensoring one can lift self-adjoint complex Hadamard matrices, and therefore, some signature matrices to higher orders. But what if one does not have an initial matrix to lift in the first place? We address this question, by recalling a standard block construction appearing in design theory and algebraic combinatorics. For a reference see the recent work of Kharaghani \emph{et al.}\ \cite{khar}.
\begin{theorem}\label{MT2}
Let $H$ be any complex Hadamard matrix of order $n$. Then there is a self-adjoint complex Hadamard matrix with constant diagonal of order $n^2$.
\end{theorem}
\begin{proof}
Let $H$ be any complex Hadamard matrix of order $n$, and let us denote its rows by $h_1,h_2,\hdots, h_n$. Consider the following block matrix, where the $(i,j)$th entry of $K$ is the block $h_j^\ast h_i$:
\beql{KX}
K=\left[
\begin{array}{cccc}
h_1^\ast h_1 & . & . & h_n^\ast h_1\\
.	& . & . & .\\
. & . & . & .\\
h_1^\ast h_n & . & . & h_n^\ast h_n\\
\end{array}
\right].
\eeq
We show that $K$ is Hadamard. Indeed: consider its $i$th row. We have
\beql{EQ1}
\sum\limits_{k=1}^{n}(h_k^\ast h_i)(h_i^\ast h_k)=\sum\limits_{k=1}^{n}h_k^\ast (h_ih_i^\ast) h_k=n\sum\limits_{k=1}^{n}h_k^\ast h_k=n^2I_n.
\eeq
Also, for rows $i\neq j$ recalling that the rows of $H$ are complex orthogonal
\beql{EQ2}
\sum\limits_{k=1}^{n}(h_k^\ast h_i)(h_j^\ast h_k)=\sum\limits_{k=1}^{n}h_k^\ast (h_ih_j^\ast) h_k=O_n,
\eeq
where $O_n$ stands for the all $0$ matrix of order $n$. Clearly, $K$ is self-adjoint by construction, and its diagonal is constant $1$.
\end{proof}
In a recent paper, Bodmann and Elwood used essentially Theorem \ref{MT2} and applied it to the Fourier matrix $H=[H]_{i,j}=\mathbf{e}^{2\pi\mathbf{i}/n(i-1)(j-1)}$ from which they concluded
\begin{corollary}[Bodmann--Elwood, \cite{BodmannPAMS}]\label{cbde}
For every $n\geq 2$ there is a self-adjoint complex Hadamard matrix of order $n^2$ with constant diagonal composed of $n$th roots of unity. Consequently there is a nontrivial $n$th root signature matrix corresponding to an equiangular $\left(n^2,n(n+1)/2\right)$ frame.
\end{corollary}
However, one might choose other complex Hadamard matrices than the Fourier matrix to obtain $n$th root signature matrices. In particular, we have the following
\begin{corollary}\label{CCC}
For every prime $p$ there is a nontrivial $p$th root signature matrix of order $4^ap^{2b}$ for all $0\leq a\leq b$ corresponding to an equiangular $\left(4^ap^{2b},2^ap^b(2^ap^b+1)/2\right)$ frame.
\end{corollary}
\begin{proof}
Indeed, as a result of Butson \cite{but} implies the existence of complex Hadamard matrices of $p$th roots of unity for every orders $2^ap^b$ such that $0\leq a\leq b$, and these matrices can be lifted to the desired order through Theorem \ref{TH}.
\end{proof}
Setting $p=3$ and $a=b=1$ in Corollary \ref{CCC} above, we immediately get the following 
\begin{corollary}\label{C}
There is a nontrivial cube root signature matrix of order $36$ leading to an equiangular $(36,21)$ frame.
\end{corollary}
A Butson-type complex Hadamard matrix of order $6$, composed of third roots of unity can be found explicitly in \cite{tao, karol}. Corollary \ref{C} answers a recent question of Bodmann \emph{et al.}\ \cite{Bodmann}. Finally, let us mention here that Theorem \ref{MT2} naturally leads to parametric families of complex Hadamard matrices. Let us recall that a complex Hadamard matrix $H$ is dephased if the first row and column of $H$ consist of numbers $1$. Clearly, every complex Hadamard matrix can be transformed into a dephased form. We have the following
\begin{proposition}\label{par}
Suppose that $H$ is a dephased complex Hadamard matrix with $m$ free parameters. Then there is a complex Hadamard matrix of order $n^2$ with $m+(n-1)^2$ free parameters, moreover there is a self-adjoint complex Hadamard matrix with constant diagonal featuring $m+(n-1)(n-2)/2$ free parameters.
\end{proposition}
\begin{proof}
Indeed, as one is free to replace the blocks of $K$ $h_j^\ast h_i$ in \eqref{KX} with $x_{i,j}h_j^\ast h_i$ for $2\leq i,j\leq n$ in Theorem \ref{MT2}, as this operation does not affect the validity of equations \eqref{EQ1} and \eqref{EQ2}. Moreover, if the unimodular variables $x_{i,j}$ are chosen in a way that $x_{i,i}=1$ and $x_{i,j}=\overline{x}_{j,i}$ for every $2\leq i,j\leq n$, then the resulting matrix is a self-adjoint complex Hadamard matrix with constant diagonal, as required. Note that the first $m$ variables are featured in the first $n$ rows of $K$, and therefore they are independent from the rest.
\end{proof}
The case $|\mu|=2$ in part (b) of Theorem \ref{TH} corresponds to self-adjoint complex Hadamard matrices with constant diagonal. However, these type of objects are somewhat rare, as the following spectral analysis shows. In particular, we have the following folklore
\begin{lemma}[Gow, \cite{RG}]\label{GOWL}
Let $H$ be a self-adjoint complex Hadamard matrix of order $n$ with constant diagonal. Then $n$ must be a square number.
\end{lemma}
\begin{proof}
Consider a self-adjoint complex Hadamard matrix of order $n$, and suppose that the number $1$ appears exactly $k$ times in its diagonal. Clearly, as $H$ is self-adjoint, its eigenvalues are $\pm\sqrt{n}$, where let us denote by $s$ the multiplicity of its positive eigenvalues. Then, we have
\[\mathrm{Tr}(H)=s\sqrt{n}+(n-s)(-\sqrt{n})=k-(n-k),\]
from which we find that $n$ is either square, or even with $H$ having zero trace, the latter case, however is not possible by our assumptions.
\end{proof}
We remark here that the proof of Lemma \ref{GOWL} implies that during the classification of self-adjoint complex Hadamard matrices of order $6$, which has been carried out in \cite{BN} successfully, one can impose the initial assumption that the diagonal of such a matrix is $[1,1,1,-1,-1,-1]$ considerably simplifying the calculations required.

Next we exhibit nontrivial complex Hadamard matrices corresponding to $|\mu|\neq 2$. The examples are to be introduced are the so-called induced complex Hadamard matrices, coming from Hadamard designs. In \cite{GC} (and, independently, but slightly later in \cite{exotic}) the following result was obtained:
\begin{lemma}[Chan--Godsil \cite{GC}]\label{CGT}
Let $U$ be a $(4m-1,2m-1,m-1)$ Hadamard design. Then by exchanging every numbers $0$ in $U$ to the unimodular complex number
\beql{1}
a=-1+\frac{1}{2m}\pm\mathbf{i}\frac{\sqrt{4m-1}}{2m}
\eeq
we obtain a complex Hadamard matrix $U_{4m-1}$.
\end{lemma}
Now we use this to exhibit complex Hadamard matrices corresponding to frames. We say that an Hadamard design $U$ is skew, if $U+U^T+I=J$ holds, where $J$ is the all $1$ matrix. Further, let us denote by $\sqrt{z}$ the principal square root of $z$. We have the following
\begin{proposition}\label{L2}
Suppose that we have a skew Hadamard design $U$ of order $n$. Then there exists a complex Hadamard matrix $H$ of order $n$ with diagonal entries $\lambda$, such that the matrix $Q:=H-\lambda I$ is self-adjoint.
\end{proposition}
\begin{proof}
Use Lemma \ref{CGT} to exhibit a complex Hadamard matrix $K=U+aU^T+aI$ where $a$ is the unimodular complex number defined in \eqref{1}. Now multiply this matrix with the unimodular number $\overline{\sqrt{a}}$ to obtain the matrix $H=\overline{\sqrt{a}}K=\overline{\sqrt{a}}U+\sqrt{a}U^T+\sqrt{a}I$.
\end{proof}
It is well-known that the Paley-type Hadamard designs of prime orders are skew, moreover it is conjectured that skew Hadamard designs exists for every order $n=4m-1$ \cite{SVy}. Therefore we have infinitely many equiangular tight frames coming from Proposition \ref{L2}. In particular, we have the following
\begin{corollary}\label{Cxd}
Suppose that we have a skew Hadamard design of order $n\geq 3$. Then there are equiangular $\left(n,(n-1)/2 \right)$ and $\left(n,(n+1)/2\right)$ frames.
\end{corollary}
Hence, if the conjecture regarding the existence of skew Hadamard matrices is true, then we have ETFs with parameters $(2k-(-1)^k,k)$ for every $k$. Renes, although used a different approach, obtained essentially the same result in \cite{R}.

An interesting question is to ask if it is possible to construct $(k^2,k)$ equiangular frames from complex Hadamard matrices. These type of frames are equivalent to SIC-POVMs \cite{grassl}, and have many applications in quantum information theory. Unfortunately, equation \eqref{kval} leads to $\mu=\sqrt{k+1}(k-2)$ which, combined with the bound $|\mu|\leq 2$, implies that $k=2$ or $3$. The case $k=2$ easily leads to the signature matrix
\[Q_4=\left[\begin{array}{cccc}
0 & 1 & 1 & 1\\
1 & 0 & \mathbf{i} & -\mathbf{i}\\
1 & -\mathbf{i} & 0 & \mathbf{i}\\
1 & \mathbf{i} & -\mathbf{i} & 0\\
\end{array}\right],\]
up to equivalence (cf.\ \cite{HP}). Case $k=3$ implies that $\mu=2$ which corresponds to, for example, the one-parameter signature matrix $-Q_9(a)$, where $Q_9(a)+I$ is the family of complex Hadamard matrices coming from Proposition \ref{par}, applied to the starting-point complex Hadamard matrix $Q_9+I$ (cf.\ formula \eqref{Q9}). One cannot but wonder what other combinatorial objects correspond to the signature matrices of $(k^2,k)$ ETFs for $k\geq 4$.
\section*{Acknowledgement}
The author is greatly indebted to the referees for their valuable comments and suggestions.


\begin{thebibliography}{11}

\bibitem{BN}
\textsc{K. Beauchamp, R. Nicoara}: Orthogonal maximal Abelian $\ast$-subalgebras of the $6\times 6$ matrices, {\it Linear Algebra and its Applications} 428 (2008) 1833-1853

\bibitem{BodmannPAMS}
\textsc{B. G. Bodmann, H. J. Elwood}: Complex Equiangular Parseval Frames and Seidel Matrices containing $p$th roots of unity, {\it Proceedings of the American Mathematical
Society}, Volume 138, Number 12, December 2010, Pages 4387-4404

\bibitem{Bodmann}
\textsc{B. G. Bodmann, V. I. Paulsen, M. Tomforde}: Equiangular tight frames from complex Seidel matrices containing cube roots of unity, {\it Linear Algebra and its Applications}, 430(2009), 396--417.

\bibitem{but}
\textsc{A.T. Butson}: Generalised Hadamard matrices, {\it Proceedings of the American Mathematical
Society} 13 (1962), 894--898. MR0142557 (26:126)

\bibitem{four}
\textsc{D. M. Duncan, Th. R. Hoffman, J. P. Solazzo}: Equiangular tight frames and fourth root Seidel matrices, {\it Linear Algebra and its Applications} 432 (2010) 2816--2823

\bibitem{RG}
\textsc{R. Gow}: Generation of Mutually Unbiased Bases as Powers of a Unitary Matrix in $2$-Power Dimensions, preprint 0703333v2

\bibitem{HP}
\textsc{R. B. Holmes, V. I. Paulsen}: Optimal frames for erasures {\it Linear Algebra Appl.} 377 31--51 (2004),

\bibitem{khar}
\textsc{W. H. Holzmann, H. Kharaghani, W. Orrick}: On the real unbiased Hadamard matrices, {\it Contemporary Mathematics, Combinatorics and Graphs}, Volume 531, 243--250. (2010)

\bibitem{KM}
\textsc{M. N. Kolountzakis, M. Matolcsi}: Complex Hadamard
matrices and the spectral set conjecture, {\it Collectanea Mathematica}, Vol. Extra, 281--291. (2006)

\bibitem{F}
\textsc{M. Fickus}: Maximally Equiangular Frames and Gauss Sums,
{\it Journal of Fourier Analysis and Applications} Volume 15, Number 3, 413-427.

\bibitem{GC}
\textsc{C. D. Godsil, A. Chan}: Type-II Matrices and Combinatorial Structures, {\it Combinatorica}, Volume 30, Number 1, 1-24,

\bibitem{God2}
\textsc{C. D. Godsil, A. Roy}: Equiangular lines, mutually unbiased bases, and spin models, {\it European Journal of Combinatorics} Volume 30 Issue 1, January, 2009

\bibitem{SVy}
\textsc{C. Koukouvinos, S. Stylianou}: On skew-Hadamard matrices, {\it Discrete Mathematics} 308, 2723--2731. (2008)

\bibitem{S1}
\textsc{P. W. H. Lemmens, J. J. Seidel}: Equiangular lines, {\it J. Algebra} 24, 494--512. (1973)

\bibitem{S2}
\textsc{J. H. van Lint, J. J. Seidel}: Equilateral point sets in elliptic geometry, {\it Indag. Math.} 28 335--348. (1966)

\bibitem{grassl}
\textsc{A. J. Scott, M. Grassl}: SIC-POVMs: A new computer study, {\it J. Math. Phys.} 51, 042203 (2010)

\bibitem{exotic}
\textsc{F. Sz\"oll\H{o}si}: Exotic Complex Hadamard matrices, and their equivalence, {\it Cryptography and communications}, Volume 2, Number 2, 187--198. (2010)

\bibitem{popa}
\textsc{S. Popa}: Orthogonal pairs of $\ast$-subalgebras in finite von Neumann algebras,
{\it J. Operator Theory}, {\bf 9}, 253--268. (1983)

\bibitem{R}
\textsc{J. M. Renes}: Equiangular tight frames from Paley tournaments
{\it Linear Algebra and its Applications} {\bf 426} 497--501. (2007)

\bibitem{tao}
\textsc{T. Tao}: Fuglede's Conjecture Is False in $5$ and Higher Dimensions, {\it Math Res. Letters}, {\bf 11}, 251--258. (2004)

\bibitem{karol}
\textsc{W. Tadej, K. \.Zyczkowski}: A concise guide to complex Hadamard matrices, {\it Open Syst. Inf. Dyn.} {\bf 13}, 133--177. (2006)

\end{thebibliography}
\end{document}